\documentclass[12pt, a4paper]{article}
\usepackage[a4paper, margin=2.5cm]{geometry} 

\usepackage{amsmath,amssymb}
\usepackage{amsthm}
\usepackage{graphicx}
\usepackage{fancybox}
\usepackage{bm}
\usepackage{wrapfig}
\usepackage{booktabs}
\usepackage{moreverb}
\usepackage{float}
\usepackage{here}
\usepackage{tikz}
\usepackage{enumerate}
\usepackage{mathtools}
\usepackage{algorithm}
\usepackage{algpseudocode}
\usepackage{xparse}

\usetikzlibrary{positioning}
\usetikzlibrary{calc}

\usepackage{enumitem}

\newenvironment{amatrix}
  {\left\langle\begin{matrix}}
  {\end{matrix}\right\rangle}

\newenvironment{smallamatrix}
  {\left\langle\begin{smallmatrix}}
  {\end{smallmatrix}\right\rangle}

\DeclareMathOperator{\GL}{GL}

\newcommand*{\gaussbinom}[3]{\genfrac{[}{]}{0pt}{}{#1}{#2}_{#3}}

\DeclareMathOperator{\cl}{cl}

\DeclareMathOperator{\MC}{MC}
\DeclareMathOperator{\MCS}{MCS}
\DeclareMathOperator{\enc}{enc}

\newcommand{\qMat}{q\mathchar`-\mathbf{Mat}^{\ncong}}

\newcommand{\INPUT}{\State \textbf{INPUT:} }
\newcommand{\OUTPUT}{\State \textbf{OUTPUT:} }

\newcommand{\qdiff}[3]{\Delta_{#3}(#1, #2)}

\theoremstyle{definition}
\newtheorem{dfn}{Definition}[section]
\newtheorem{rem}[dfn]{Remark}
\newtheorem*{note*}{Notation}
\newtheorem{eg}[dfn]{Exapmle}

\theoremstyle{plain}
\newtheorem{thm}[dfn]{Theorem}
\newtheorem{lem}[dfn]{Lemma}
\newtheorem{prop}[dfn]{Proposition}
\newtheorem{cor}[dfn]{Collorary}

\setlist[enumerate]{leftmargin=15mm}

\renewcommand{\thealgorithm}{\Alph{algorithm}}
\makeatletter
\newenvironment{breakablealgorithm}
  {\begin{center}
     \refstepcounter{algorithm}
     \centerline{\rule{\linewidth}{0.8pt}} \kern-5pt
     \vspace{2pt}
     \renewcommand{\caption}[2][\relax]{{\raggedright\textbf{Algorithm~\thealgorithm} ##2\par}\ifx\relax##1\relax
       \else
         \addcontentsline{loa}{algorithm}{\protect\numberline{\thealgorithm}##1}\fi
       \vspace{-9pt}
       \kern1pt \centerline{\rule{\linewidth}{0.2pt}} \kern-5pt
       \vspace{3pt}
     }
  }
  {
    \vspace{-2pt}
    \kern-5pt \centerline{\rule{\linewidth}{0.2pt}}
   \end{center}
  }
\makeatother
  \makeatletter
\def\thesis#1{\def\@thesis{#1}}
\def\studentid#1{\def\@studentid{#1}}
\def\university#1{\def\@university{#1}}
\def\school#1{\def\@school{#1}}
\def\department#1{\def\@department{#1}}
\def\course#1{\def\@course{#1}}
\def\supervisor#1{\def\@supervisor{#1}}
\def\address#1{\def\@address{#1}}
\def\email#1{\def\@email{#1}}
\def\university#1{\def\@university{#1}}
\def\affiliation#1{\def\@affiliation{#1}} 
\def\@maketitle{
  \thispagestyle{empty}
  \begin{flushright}
  \end{flushright}
  \begin{center}
    \vspace{5mm}
    {\LARGE \@title} \\ \vspace{5mm}
    {\@author}\\
    \vspace{5mm}
    {$^{1}$Research and Education Institute for Semiconductors and Informatics,\\
  Kumamoto University,\\
  2-39-1, Kurokami, Kumamoto 860-8555, Japan,}\\
    \vspace{5mm}
    {$^{2, 3}$Department of Semiconductor, Computer Science and Applied Mathematics
  Kumamoto University,\\
  2-39-1, Kurokami, Kumamoto 860-8555, Japan}\\
  \vspace{5mm}
  {\@email}\\
  \vspace{5mm}
  {\@date}\\ \end{center}
}
 \makeatother
 \thesis{Thesis}
\title{On the one-dimensional extensions of $q$-matroids}
\course{Course of Applied Mathematics}
\email{\texttt{$^{1}$imamura.math.engr@gmail.com},\\
\texttt{$^{2}$k.shinya.math.eng@gmail.com}, \\
\texttt{$^{3}$keisuke@kumamoto-u.ac.jp}.}
\university{Kumamoto University}
\author{Koji Imamura$^1$ \quad Shinya Kawabuchi$^{2, *}$ \quad Keisuke Shiromoto$^{3}$}

\date{\today}
 \begin{document}

  \maketitle
  \vspace{5mm}

  \abstract
In this paper we introduce a $q$-analogue of the single-element extensions of matroids for $q$-matroids, which we call one-dimensional extensions.
To enumerate such extensions, we define a $q$-analogue of modular cuts and define a certain function which we call a modular cut selector.
It assigns each newly appearing one-dimensional subspace to a modular cut. 
By using these notion, we prove the one-to-one correspondence between the one-dimensional extensions and the modular cut selectors.
Furthermore, we define the canonnical representatives of the isomorphic class of the $q$-matroids, which enable us to enumerate non-isomorphic $q$-matroids without the paiwise isomorphism testing.
As an application, we develop a classification algorithm for $q$-matroids, 
and classify all the $q$-matroids on ground spaces over $\mathbb{F}_2$ and $\mathbb{F}_3$ of dimension $4$ and $5$ respectively. 
We also determine some $5$-dimensional $q$-matroids related to the $q$-Fano plane, 
which is the $q$-analogue of the Fano plane, over $\mathbb{F}_2$.   \vspace{5mm}

 
  \section{Introduction} \label{Introduction}
  Matroid theory is a one of the study of combinatorial structures which abstract the notion of linear independence in linear algebra.
Typically, matroids are constructed from various combinatorial and algebraic structures such as graphs and matrices \cite{Oxley, Welsh}.
One of the central topics in matroid theory is the enumeration and classification problems
which have applications to verification for conjectures related to matroids,
such as the points-lines-planes conjecture by Welsh and Seymour \cite{Seymour, MEIG} and
Terao's free conjecture in the theory of hyperplane arrangements \cite{GRSMA}.
In \cite{MNE}, matroids of size $n \leq 9$ have been classified for all possible ranks.
In general, it is quite challenging to classify matroids of an arbitrary rank when $n > 9$,
and so partial results are obtained in \cite{LSAM12, CCG, MEIG}, where matroids of size $n$ and rank $k$
have been classified when $n \leq 12$ and $k \leq 3$, and in the case of $(n, k) = (10, 4)$.

A common approach for enumerating classical matroids relies on the enumeration of \emph{single-element extensions}
which are matroids obtained by adding a new element to a given matroid.
It is known that there is a bijection between certain subcollections of flats of a matroid, known as a \emph{modular cut}, and its single-element extensions.
For more detailed enumeration algorithms of matroids, see \cite{CCG, MNE, MEIG}.

Certain combinatorial structures and their properties admit a $q$-analogue which extends the concept from finite sets to finite vector spaces.
One of the most typical examples to grasp the notion of $q$-analogue is the Gaussian binomial coefficients, or the $q$-binomial coefficients.
Given two non-negative integers $k$ and $n$, they count the number of the $k$-dimensional subspaces of an $n$-dimensional space over a field with $q$ elements
as the ordinal binomial coefficients count the number of $k$-subsets of an $n$-set.
For more details of $q$-analogue, see, for instance, \cite[Chapter 1]{Stanley}.
The notion of $q$-matroids is a $q$-analogue of matroids.
Its study originated in Crapo \cite{CrapoPhd} was re-found by Jurrius and Pellikaan \cite{DQM},
where they demonstrated that a certain class of rank-metric codes, called \emph{Gabidulin codes}, induces $q$-matroids.

$q$-Matroids also have applications to $q$-Steiner systems which are a $q$-analogue of Steiner systems.
One of the major topics in the study of $q$-Steiner systems is the question about the existence of a particular $q$-Steiner system, known as a $q$-Fano plane.

Byrne \textit{et al.} showed that $q$-Steiner systems induce $q$-matroids \cite{CNMD}.
Consequently, the classification of $q$-matroids of dimension $7$ and rank $3$ reveals whether the $q$-Fano plane exists or not.
The classification of $q$-matroids has been completed up to dimension $3$ (see the appendix of \cite{DSOM});
however, no systematic enumeration methods for $q$-matroids are known.
In this paper, we propose such enumeration for them.

In Section \ref{Preliminaries}, we give an overview of the single element extensions of matroids and fundamental properties of $q$-matroids.
In Section \ref{ODEQM}, 
we define one-dimensional extensions and modular cuts as the $q$-analogues of those of classical matroids.
We also define a certain class of functions, which we call modular cut selectors, 
and prove that they are equivalent to one-dimensional extensions.
In Section \ref{Results} we develop a classification algorithm for $q$-matroids using modular cut selectors. 
We also introduce the canonical representaives of the isomorphic class of the $q$-matroids 
which is a $q$-analogue of canonical matroid in \cite{MEIG}.
In Section \ref{q-FanoPlane}, we present some $q$-matroids related to the $q$-Fano plane over $\mathbb{F}_2$. 
  \section{Preliminaries} \label{Preliminaries}
  \subsection{Matroids} \label{Matroids}

Hereafter, for a function $f$ on a set $X$ and a subset $Y$ of $X$, we denote by $f|_{Y}$ the restriction of $f$,
that is, $f|_{Y}(x) \coloneqq f(x)$ for all $x\in Y$.
Given a set $X$, we also denote by $2^X$ the collection of all subsets of $X$.

\begin{dfn} \label{def:matroid}
  Let $E$ be a finite set and $r$ be a non-negative integer function on $2^E$.
  The ordered pair $M = (E, r)$ is a \emph{matroid} if the following properties hold:
  \begin{enumerate}[label=\normalfont{(r\arabic*)}]
    \item if $X \in 2^E$, then $0 \leq r(X) \leq |X|$;
    \item if $X, Y \in 2^E$ satisfy $X \subseteq Y$, then $r(X) \leq r(Y)$;
    \item if $X, Y \in 2^E$, then $r(X \cup Y) + r(X \cap Y) \leq r(X) + r(Y)$.
  \end{enumerate}
  Then, we call $E$ and $|E|$ the \emph{ground set} and the \emph{size} of $M$, respectively.
  The function $r$ is called the \emph{rank function} of $M$,
  and $r(X)$ is referred to as the \emph{rank} of $X$.
  Moreover, $r(E)$ is the \emph{rank} of $M$.
\end{dfn}

We introduce some terminology for matroids, following \cite{Oxley}.

\begin{dfn} \label{def:restriction}
  For any matroid $M=(E, r)$, the \emph{restriction} of $M$ to a subset $E' \subseteq E$, denoted by $M|E'$, is given by $(E', r|_{2^{E'}})$.
\end{dfn}

\begin{dfn} \label{def:extension}
  A matroid $N = (E_N, r_N)$ is called an \emph{extension} of a matroid $M = (E_M, r_M)$ if $N|E_M = M$.
  In particular, if $|E_N| = |E_M|+1$, then $N$ is called a \emph{single-element extension} of $M$.
\end{dfn}

\begin{dfn} \label{def:flat}
  Let $M=(E, r)$ be a matroid.
  A subset $F$ of $E$ is called a \emph{flat} of $M$ if
  $r(F \cup \{e\}) = r(F) + 1$ for all $e  \in E\setminus F$.
  We denote by $\mathcal{F}_M$ the collection of all flats of $M$.
\end{dfn}

\begin{dfn} \label{def:modular_pair}
  Let $M=(E, r)$ be a matroid. A pair $(F_1, F_2)$ of two flats of $M$ is called a \emph{modular pair} of $M$ if
  \begin{equation*}
    r(F_1 \cup F_2) + r(F_1 \cap F_2) \leq r(F_1) + r(F_2).
  \end{equation*}
\end{dfn}

\begin{dfn} \label{def:closure_operator}
  For any matroid $M=(E, r)$, the function $\cl_M \colon 2^E \to 2^E$ defined by
  \begin{equation*}
    \cl_M(X) \coloneqq \{e \in E \mid r(X \cup \{e\}) = r(X)\}
  \end{equation*}
  for all $X \in 2^E$ is called the \emph{closure operator} of $M$.
\end{dfn}

In classical matroid theory, the classification problem of matroids involves the enumeration of non-isomorphic matroids of given size $n$ and rank $k$.
The classification problem goes back to the study by Blackburn, Crapo, and Higgs \cite{CCG} in 1973,
and has been solved up to $n=9$ in \cite{MNE}, whereas partial classifications are conducted for $n=10, 11$, and $12$ in \cite{MEIG, LSAM12}.

Notably, the approaches in \cite{SEEM, MNE, MEIG} involve enumerating single-element extensions of non-isomorphic matroids of size $n-1$
to classify non-isomorphic matroids of size $n$, where they adopt different approaches to the isomorphism test.
Here, we focus on the common method of constructing single-element extensions in \cite{SEEM, MNE, MEIG}.

First, we describe the single-element extensions from the perspective of flats (see also \cite[Section 7.2]{Oxley}).
Suppose that $N=(E_N, r_N)$ is a single-element extension of a matroid $M=(E_M, r_M)$ by a new element $e$.
Then, the collection of the flats in $M$ is equal to $\{F \cap E_M \mid F \in \mathcal{F}_N\}$.
Conversely, all flats of $N$ are written as either $F$ or $F \cup \{ e \}$ for some flat $F$ of $M$.
The concept of modular cuts below is based on the observation that every flat $F$ of $M$ with $r_N(F \cup e) = r_N(F)$ satisfies that $F \cup \{e\}$ is a flat of $N$.

\begin{dfn}[\cite{SEEM, Oxley}] \label{def_modular_cut}
  For a matroid $M=(E_M, r_M)$, a subset $\mathcal{M}$ of the family of flats $\mathcal{F}_M$ is called a \emph{modular cut}
  if it satisfies the following conditions:
  \begin{enumerate}[label=\normalfont{(m\arabic*)}]
    \item if $F \in \mathcal{M}$ and $F'$ is a flat containing $F$, then $F' \in \mathcal{M}$;
    \item if $F_1, F_2 \in \mathcal{M}$ and $(F_1, F_2)$ is a modular pair, then $F_1 \cap F_2 \in \mathcal{M}$.
  \end{enumerate}
\end{dfn}

\begin{thm}[\cite{SEEM, Oxley}]
  If $N = (E_N, r_N)$ is a single-element extension of a matroid $M=(E_M, r_M)$, where $E_N = E_M \cup \{e\}$,
  then the following set forms a modular cut of $M$:
  \begin{equation*}
    \{F \in \mathcal{F}_M \mid F \cup \{e\} \in \mathcal{F}_N \text{ and } r_N(F \cup e) = r_N(F)\}.
  \end{equation*}
\end{thm}

Conversely, by reversing the construction in the theorem above, each modular cut gives rise to a unique extension.

\begin{thm}[\cite{SEEM, Oxley}]
  Let $M=(E_M, r_M)$ be a matroid, $\mathcal{M}$ be a modular cut of $M$ and $E_N = E_M \cup \{e\}$. Define a function $r' \colon 2^{E_N} \to \mathbb{Z}_{\geq 0}$ for all $X \in 2^{E_M}$ as follows:
  \begin{align*}
    r'(X) = r_M(X), \quad \textit{and} \quad
    r'(X+e) =
    \begin{cases}
      r_M(X)   & \text{if } \cl_M(X) \in \mathcal{M},\\
      r_M(X)+1 & \text{if } \cl_M(X) \notin \mathcal{M}.
    \end{cases}
  \end{align*}
  Then, $(E_N, r')$ is a single-element extension of $M$.
\end{thm}
   \subsection{$q$-Matroids} \label{q-Matroids}
  Throughout this paper, fix a prime power $q$ and a non-negative integer $n$.
We denote by $\mathbb{F}_q$ the finite field with $q$ elements.
$0$ represents the $0$-dimensional vector space over $\mathbb{F}_q$.
If $X$ is a vector space over $\mathbb{F}_q$, we denote by $\mathcal{L}(X)$ the set of all subspaces of $X$,
by $[X]_q$ the set of all one-dimensional subspaces of $X$,
and define $\qdiff{B}{A}{q} \coloneqq [B]_q \setminus [A]_q$ for all $A, B \in \mathcal{L}(X)$.
We also denote by $\gaussbinom{X}{k}{q}$ the collection of all $k$-dimensional subspaces in $X$ and 
by $\gaussbinom{n}{k}{q}$ the number of $k$-dimensional subspaces in an $n$-dimensional vector space over $\mathbb{F}_q$.
If $A$ is a subspace of a vector space $B$, we write $A \leq B$. 
For any function $f$ on a vector space $E$ over $\mathbb{F}_q$, we
denote by $\hat{f}$ the function on $\mathcal{L}(E)$ which naturally extends $f$,
 that is, $\hat{f}(X)$ is the
image of $X$ under $f$ for all $X \in \mathcal{L}(E)$.

\begin{dfn}[\cite{DQM}] \label{def:q-matroid}
  Let $E$ be a finite-dimensional vector space over $\mathbb{F}_q$ and $r$ be a non-negative integer function on $\mathcal{L}(E)$.
  The ordered pair $M = (E, r)$ is a $q$-\emph{matroid} if the following properties hold:
  \begin{enumerate}[label=\normalfont{(R\arabic*)}]
      \item if $X \in \mathcal{L}(E)$, then $0 \leq r(X) \leq \dim(X)$; \label{q-mat_boundedness}
      \item if $X, Y \in \mathcal{L}(E)$ satisfy $X \leq Y$, then $r(X) \leq r(Y)$; \label{q-mat_monotonicity}
      \item if $X, Y \in \mathcal{L}(E)$, then $r(X + Y) + r(X \cap Y) \leq r(X) + r(Y)$. \label{q-mat_submodularity}
  \end{enumerate}
We call $E$ and $\dim E$ the \emph{ground space} and the \emph{dimension} of $M$, respectively.
  The function $r$ is called the \emph{rank function} of $M$,
  and $r(X)$ is referred to as the \emph{rank} of $X$.
  Moreover, $r(E)$ is the \emph{rank} of $M$.
\end{dfn}

\begin{eg}[\cite{DQM}]
  Let $E$ be a $n$-dimensional space over $\mathbb{F}_q$ and $k$ be an integer with $0 \leq k \leq n$. We define the function $r:\mathcal{L}(E) \to \mathbb{Z}_{\geq 0}$ by $r(X) \coloneqq \min \{\dim X, k\}$. Then, $M=(E, r)$ be a $q$-matroid, which is called the \emph{uniform $q$-matroid} denoted by $U_{k,n}$.
\end{eg}

\begin{lem} \label{rXAB}
  Let $(E, r)$ be a $q$-matroid. For all $A, B, X \in \mathcal{L}(E)$ and $x \in [E]_q$, the following statements hold.
  \begin{enumerate}[label=\normalfont{(\arabic*)}]
    \item $r(A + x) \leq r(A) + 1$.\label{AxA}
\item If $A \subseteq B$ and $r(A + x) = r(A)$, then $r(B + x) = r(B)$. \label{AxABxB}
    \item If $r(X + A) = r(X + B) = r(X)$, then $r(X + A + B) = r(X)$.\label{XABX}
  \end{enumerate}
\end{lem}

\begin{proof}
  For \ref{AxA}, see \cite[Lemmas 2.3]{DQM}.
  The statement \ref{AxABxB} is the contraposition of \cite[Lemma 3.2]{CNMD}.
  We check \ref{XABX}.
  Since $X \leq (X+A) \cap (X+B)$, it follows from \ref{q-mat_monotonicity} and \ref{q-mat_submodularity} that
  \begin{align*}
    r(X + A + B) + r(X) &\leq r(X + A + X + B) + r((X+A) \cap (X+B))\\
                        &\leq r(X + A) + r(X + B)\\
                        &= 2r(X).
  \end{align*}
  Therefore, $r(X+A+B) \leq r(X)$ holds, which implies $r(X+A+B) = r(X)$ by \ref{q-mat_monotonicity}. Hence (4) holds.
\end{proof}

\begin{dfn} \label{def:q-mat_closure_operator}
  For any $q$-matroid $M = (E, r)$, the function $\cl_M \colon \mathcal{L}(E) \to \mathcal{L}(E)$ defined by
  \begin{equation*}
    \cl_M(X) \coloneqq \sum \{x \in [E]_q \mid r(X + x) = r(X)\}
  \end{equation*}
  for all $X \in \mathcal{L}(E)$ is called the \emph{closure operator} of $M$, and we say $\cl_M(X)$ is the $closure$ of $X$.
\end{dfn}

\begin{prop}[\cite{CNQC}]\label{claxiom}
  Let $M = (E, r)$ be a $q$-matroid. For all $X, Y \in \mathcal{L}(E)$ and $x, y \in [E]_q$, the following properties hold:
  \begin{enumerate}[label=\normalfont{(CL\arabic*)}]
    \item $X \leq \cl_M(X)$; \label{q-mat_cl_increasing}
    \item if $X \leq Y$, then $\cl_M(X) \leq \cl_M(Y)$; \label{q-mat_cl_monotonicity}
    \item $\cl_M(X) = \cl_M(\cl_M(X))$; \label{q-mat_cl_idempotence}
    \item if $y \in \qdiff{\cl_M(X + x)}{\cl_M(X)}{q}$, then $x \in [\cl_M(X + y)]_q$. \label{q-mat_cl_MacLane-Steinitz_exchange}
  \end{enumerate}
\end{prop}

\begin{prop} \label{proprcl}
  Let $M = (E, r)$ be a $q$-matroid. Then, for all $X, Y \in \mathcal{L}(E)$, we have
  \begin{equation*}
    r(X + Y) = r(X + \cl_M(Y)) = r(\cl_M(X) + \cl_M(Y)).
  \end{equation*}
\end{prop}
\begin{proof}
  By Proposition~\ref{rXrclX}, we have $r(X + Y) = r(\cl_M(X + Y))$. From \ref{q-mat_monotonicity} and \ref{q-mat_cl_increasing},
  it follows that
  \begin{equation*}
    r(X + Y) \leq r(X + \cl_M(Y)).
  \end{equation*}
  Additionally, from Proposition~\ref{rXrclX}, \ref{q-mat_submodularity}, and \ref{q-mat_cl_increasing},
  \begin{align*}
    r(X + Y) + r(Y) &= r(X + Y) + r(\cl_M(Y)) \\
                    &\geq r((X + Y) \cap \cl_M(Y)) + r(X + \cl_M(Y)).
  \end{align*}
  Noting that, $r((X + Y) \cap \cl_M(Y)) \geq r(Y)$ from \ref{q-mat_monotonicity} and \ref{q-mat_cl_increasing},
  we obtain
  \begin{equation*}
    r(X + Y) \geq r(X + \cl_M(Y)).
  \end{equation*}
  Hence, $r(X + Y) = r(X + \cl_M(Y))$.

  Now, by replacing $X$ with $\cl_M(X)$ and repeating the same argument, we have $r(X + \cl_M(Y)) = r(\cl_M(X) + \cl_M(Y))$.
\end{proof}

\begin{dfn} \label{def:q-mat_flat}
  Let $M = (E, r)$ be a $q$-matroid.
  A subspace $F \in \mathcal{L}(E)$ is called a \emph{flat} of $M$ if $r(F + x) > r(F)$ for every $x \in \qdiff{E}{F}{q}$.
  Furthermore, $\mathcal{F}_M$ denotes the set of all flats of $M$.
\end{dfn}

\begin{dfn} \label{def:q-mat_cover}
  Let $\mathcal{A}$ be a subset of $\mathcal{L}(E)$ and $A, B \in \mathcal{A}$.
  We say $B$ \emph{covers} $A$ in $\mathcal{A}$ if $A \leq B$ and for any $C \in \mathcal{A}$ such that $A \leq C \leq B$, then either $C = A$ or $C = B$.
\end{dfn}

\begin{prop}[\cite{CNQC}]
  If $M=(E, r)$ is a $q$-matroid, then $\mathcal{F}_M$ satisfies the following properties:
  \begin{enumerate}[label=\normalfont{(F\arabic*)}]
      \item $ E \in \mathcal{F}_M $;
      \item if $ F_1, F_2 \in \mathcal{F}_M$, then $ F_1 \cap F_2 \in \mathcal{F}_M $;
      \item if $ F \in \mathcal{F}_M $ and $ e \in \qdiff{E}{F}{q}$, there uniquely exists $F' \in \mathcal{F}_M$ containing $e$ such that $F'$ covers $F$ in $\mathcal{F}_M$.
  \end{enumerate}
\end{prop}

\begin{prop}[\cite{CNQC}]\label{clFF}
  Let $M = (E, r)$ be a $q$-matroid. For all $F \in \mathcal{F}_M$,
  \begin{equation*}
    \cl_M(F) = F.
  \end{equation*}
\end{prop}

\begin{prop}[\cite{CNQC}]\label{rXrclX}
  Let $M = (E, r)$ be a $q$-matroid.
  For all $X \in \mathcal{L}(E)$, its closure $\cl_M(X)$ is a flat of $M$ satisfying $r(\cl_M(X)) = r(X)$.
\end{prop}

\begin{dfn} \label{def:q-mat_restriction}
  The \emph{restriction} of a $q$-matroid $M = (E, r)$ to a subspace $E' \leq E$, denoted by $M|E'$, is given by $(E', r|_{\mathcal{L}(E')})$.
\end{dfn}
We note that restrictions of $q$-matroids are $q$-matroids \cite{DQM}.

\begin{prop}[\cite{CNQC}]\label{clF}
  Let $M = (E, r)$ be a $q$-matroid.
Then, we have
\begin{equation*}
    \mathcal{F}_M = \{\cl_M(X) \mid X \in \mathcal{L}(E)\}.
  \end{equation*}
\end{prop}

\begin{thm}\label{cor_rest}
 Let $N=(E_N, r_N)$ be a $q$-matroid and $E_M \leq E_N$.
  If $M=(E_M, r_M)$ is the restriction of $N$ to $E_M$, then we have
  \begin{equation*}
    \cl_M(X) = \cl_N(X) \cap E_M
  \end{equation*}
  for all $X \in \mathcal{L}(E_M)$, and
  \begin{equation*}
    \mathcal{F}_M = \{F \cap E_M \mid F \in \mathcal{F}_N\}.
  \end{equation*}
\end{thm}

\begin{proof}
We take $X \in \mathcal{L}(E_M)$. By the definition of closure operator and restriction,
  \begin{equation*}
    \cl_M(X) = \cl_M(X) \cap E_M \leq \cl_N(X) \cap E_M.
  \end{equation*}
  For all $x \in [\cl_N(X) \cap E_M]_q = [\cl_N(X)]_q \cap [E_M]_q$, we have $r_N(X + x) = r_N(X)$ by
Proposition~\ref{rXrclX}.
  Hence, we have $[\cl_N(X) \cap E_M]_q \subseteq [\cl_M(X)]_q$, which implies that $\cl_N(X) \cap E_M\leq\cl_M(X)$.

  Applying Proposition~\ref{clF} twice, we also obtain
  \begin{equation*}
    \mathcal{F}_M = \{\cl_N(X) \cap E_M \mid X \in \mathcal{L}(E_N)\}=\{F \cap E_M \mid F \in \mathcal{F}_N\}.\qedhere
  \end{equation*}
\end{proof}

\begin{dfn}
  Let $M=(E, r)$ be a $q$-matroid and $X$ be a subspace of $E$. 
  If $r(X) = \dim{X} = r(E)$, we call $X$ a \emph{basis} of $M$.
  We denote by $\mathcal{B}_M$ the collection of all the bases of $M$.
\end{dfn}

If a finite dimensional vector space $E$ over $\mathbb{F}_q$ and its subspaces $\mathcal{F}$ satisfy (F1), (F2), and (F3),
there uniquely exists a $q$-matroid whose flats are $\mathcal{F}$. 
Hence, we identify the rank function with the flats.
There are also such correspondence among its closure operator, bases, flats, and rank function.
We call them the \emph{cryptomorphisms} of $q$-matroids (see \cite{CNQC}).

\begin{dfn} \label{def:q-mat_isomorhpic}
  Let $M_1 = (E_{1}, r_{1})$ and $M_2 = (E_{2}, r_{2})$ be $q$-matroids, where $E_1$ and $E_2$ are over $\mathbb{F}_q$.
  If there exists an $\mathbb{F}_q$-isomorphism $g \colon E_{1} \to E_{2}$ such that $r_1 = r_2 \circ \hat{g}$,
  we say $M_1$ is \emph{isomorphic} to $M_2$ and write $M_1 \cong M_2$.
\end{dfn}

If $M_1 = (E_1, r_1)$ is a $q$-matroid and $g$ is an $\mathbb{F}_q$-isomorphism $g \colon E_1 \to E_2$, it is easy to see that $(E_1, r_1 \circ \hat{g})$ is a $q$-matroid which is isomorphic to $M_1$. 
It is also straightforward that $\cong$ is an equivalence relation.

\begin{dfn} \label{def:q-mat_dual}
  Let $M = (E, r)$ be a $q$-matroid. The ordered pair $M^{\ast} = (E, r^{\ast})$ is called the \emph{dual} of $M$,
  where $r^{\ast}$ is the non-negative integer function on $\mathcal{L}(E)$ defined as follows: for each $X \in \mathcal{L}(E)$,
  \begin{equation*}
    r^{\ast}(X) \coloneqq \dim(X) - r(M) + r(X^{\perp}),
  \end{equation*}
  where $X^\perp$ is the orthogonal complement of $X$ in $E$.
\end{dfn}
As shown in \cite{DQM}, duals of $q$-matroids are $q$-matroids.

\begin{prop}[\cite{RMCQP}]\label{iso_dual}
  $M_1 \cong M_2$ if and only if ${M_1}^{\ast} \cong {M_2}^{\ast}$.
\end{prop}

\begin{dfn} \label{def:q-mat_extension}
  A $q$-matroid $N=(E_N, r_N)$ is called an \emph{extension} of a $q$-matroid $M=(E_M, r_M)$ if $N|E_M=M$.
  In particular, if $\dim E_N  =  \dim E_M + 1$, then $N$ is called a \emph{one-dimensional extension} of $M$ by $e \in \qdiff{E_N}{E_M}{q}$.
\end{dfn}

\begin{prop}
   Suppose $M_1 \cong M_2$. Let $E_{1}'$ and $E_{2}'$ be vector spaces over $\mathbb{F}_q$ with $\dim E_{1}' = \dim E_{2}'$ such that $E_1 \leq E_1'$ and $E_2 \leq E_2'$.
   If $M'_1=(E'_1, r'_1)$ is an extension of $M_1$ to $E_1'$, there exists an extension $M'_2 = (E_2', r'_2)$ of $M_2$ to $E_2'$
   such that $M'_1 \cong M'_2$.
\end{prop}

\begin{proof}
  Take an $\mathbb{F}_q$-isomorphism $g \colon E_2 \to E_1$ such that $r_2 = r_1 \circ \hat{g}$.
  We note that there exists an $\mathbb{F}_q$-isomorphism $f \colon E'_2 \to E'_1$ such that $\hat{f}|_{\mathcal{L}(E_2)} = \hat{g}$.
  We set a function $r'_2 \coloneqq r'_1 \circ \hat{f}$.
  Then, $r'_2|_{\mathcal{L}(E_2)} = (r'_1|_{\mathcal{L}(E_1)}) \circ (\hat{f}|_{\mathcal{L}(E_2)}) = r_1 \circ \hat{g} = r_2$.
  Hence, $M'_2=(E'_2, r'_2)$ is an extension of $M_2$ and $M'_1 \cong M'_2$.
\end{proof}
 
  \section{Extensions of $q$-matroids}\label{ODEQM}
  In this section, we show that one-dimensional extensions are obtained by selecting modular cuts properly.
Basically we follow the argument in \cite[Section 7.2]{Oxley} to prove each result,
where we remark slight differences caused by taking a $q$-analogue if necessary.
Hereafter, set $E'$ as a $\dim E + 1$ dimensional vector space containing $E$.

\subsection{Modular Cuts of $q$-matroids}

All one-dimensional extensions of a $q$-matroid $M=(E, r_M)$ to $E$ could, in principle,
be obtained by taking the restriction of all $q$-matroids on $E'$ and checking whether they coincide with $M$.
However, it is impractical to list explicitly all $q$-matroids on a given ground space.
Therefore, we develop a method to construct all the one-dimensional extensions of $M$ using certain flats of $M$.
We begin by examining the behavior of the rank of flats in one-dimensional extensions of $M$.

We take $F \in \mathcal{F}_M$. We assume $N=(E', r_N)$ is a one-dimensional extension of $M$ by $e$. Then, by Theorem~\ref{cor_rest} and Proposition~\ref{proprcl},
one of the following properties holds:
\begin{itemize}
  \item $F + e$ is a flat of $N$ and $r_N(F + e) = r_N(F)$;
  \item $F + e$ is a flat of $N$ and $r_N(F + e) = r_N(F) + 1$;
  \item $F + e$ is not a flat of $N$.
\end{itemize}

Now we observe a $q$-analogue of \cite[Lemma 7.2.2]{Oxley}. For this purpose, we employ the $q$-analogue of the modular pairs.

\begin{dfn} \label{def:q-mat_modular_pair}
  Let $M=(E, r)$ be an $q$-matroid.
  A pair $(F_1, F_2)$ of two flats of $M$ is called a \emph{modular pair} of $M$ if
  \begin{equation*}
    r(F_1 + F_2) + r(F_1 \cap F_2) = r(F_1) + r(F_2).
  \end{equation*}
\end{dfn}

\begin{lem} \label{lem:modularcut_1}
  Let $N=(E', r_N)$ be a one-dimensional extension of $M$ by $e$ and $\mathcal{M}$ be a subset of $\mathcal{F}_M$.
  If $\mathcal{M}= \{F \in \mathcal{F}_M \colon F + e \in \mathcal{F}_M \textrm{\;and\;} r_N(F + e) = r_N(F)\}$,
  then $\mathcal{M}$ has the following properties:
  \begin{enumerate}[label=\textnormal{(M\arabic*)}]
    \item if $F \in \mathcal{M}$ and $F'$ is a flat of $M$ containing $F$, then $F' \in \mathcal{M}$; \label{M1}
    \item if $F_1, F_2\in \mathcal{M}$ and $(F_1, F_2)$ is a modular pair of $M$, then $F_1 \cap F_2 \in \mathcal{M}$. \label{M2}
  \end{enumerate}
\end{lem}

\begin{proof}
  We consider a $q$-analogue of the proof of \cite[Lemma 7.2.2]{Oxley}.
  To derive the equation corresponding to \cite[Equation~(7.1)]{Oxley},
  it suffices to observe the fundamental fact that $(F_1 + e) \cap (F_2 + e) = (F_1 \cap F_2) + e$, which follows from linear algebra.
  The remaining part is completed by replacing the operation `$\cup$'  with `$+$' and by applying \ref{q-mat_monotonicity}, Lemmas~\ref{rXAB}\ref{AxABxB}, and Theorem~\ref{cor_rest}.
\end{proof}

For any $q$-matroid $M$, if a subset $\mathcal{M}$ of $\mathcal{F}_M$ satisfies \ref{M1} and \ref{M2},
we call $\mathcal{M}$ a \textit{modular cut} of $M$, and denote by $\MC_M$ the collection of all modular cuts of $M$.
Noting that $|\qdiff{E'}{E}{q}| = 1$ does \textit{not} generally hold,
we observe a key difference from the single-element extension in classical matroid theory.
When we take $e, e' \in \qdiff{E'}{E}{q}$ and a flat $F \in \mathcal{F}_M$, it is straightforward from a similar argument before Lemma~\ref{lem:modularcut_1}
that either of the following properties holds:
\begin{itemize}
  \item $F + e$ (or $F + e'$) is a flat of $N$ but $F + e'$ (or $F + e$) is not a flat of $N$;
  \item $F + e$ and $F + e'$ both are flats of $N$;
  \item neither $F + e$ nor $F + e'$ is a flat of $N$.
\end{itemize}

\begin{lem}\label{lem_QM}
  Let $N=(E', r_N)$ be a one-dimensional extension of a $q$-matroid $M$ on $E$.
  The mapping which assigns
  \begin{equation*}
    \{F \in \mathcal{F}_M \mid F + e \in \mathcal{F}_N \text{ and } r_N(F + e) = r_N(F) \}
  \end{equation*}
  for all $e \in \qdiff{E'}{E}{q}$ defines a function $\mu_N \colon \qdiff{E'}{E}{q} \to \MC_M$.
  Then $\mu_N$ satisfies the following property:
  \begin{enumerate}[label=\textnormal{(QM)}]
    \item Let $e_1, e_2 \in \qdiff{E'}{E}{q}$ and $F \in \mu_N(e_1)$. Then, $F + e_1 = F + e_2$ if and only if $F \in \mu_N(e_2)$. \label{QM}
  \end{enumerate}
\end{lem}

\begin{proof}
  If $N$ is an extension of $M$ by $e$, $N$ is also regarded as extension by any other element $e'$ in $\qdiff{E'}{E}{q}$.
  Hence, the codomain of $\mu_N$  is $\MC_M$.

  If $F + e_1 = F+e_2$, then $r_N(F + e_2) = r_N(F + e_1) = r_N(F)$, and so $F \in \mu_N(e_2)$.
  Conversely, by the definition of $\mu_N$ and $\cl_N$, if $F \in \mu_N(e_2)$, then we have that \quad
  $\cl_N(F) = F + e_1 = F + e_2$.
  Therefore, $\mu_N$ satisfies \ref{QM}.
\end{proof}

\begin{eg}
  Let $E'$ be the row space $\begin{smallamatrix}100\\010\\001\end{smallamatrix}$ over $\mathbb{F}_2$ and
  the function $r_N \colon \mathcal{L}(E') \to \mathbb{Z}_{\geq 0}$ satisfy $r_N(X) \coloneqq \max\{\dim{X}, 1\}$ for all $X \in \mathcal{L}(E')$.
  Then $(E', r_N)$ is a uniform $q$-matroid $U_{1, 3}$.
  We also let $E = \begin{smallamatrix}100\\011 \end{smallamatrix}$ and $M$ be the restriction of $N$ to $E$.
  Then,
  \begin{equation*}
    \mathcal{F}_M = \left\{\langle 000 \rangle, \begin{amatrix} 100\\011 \end{amatrix}\right\}
    \quad 
    \MC_M=\left\{
      \emptyset,
      \left\{\begin{amatrix}100\\011 \end{amatrix}\right\}, 
      \left\{\langle 000 \rangle, \begin{amatrix} 100\\011 \end{amatrix}\right\}
    \right\}.
  \end{equation*}
  Note that $\qdiff{E'}{E}{q} = \{\langle 001 \rangle, \langle 010 \rangle, \langle 101 \rangle, \langle 110 \rangle\}$ and the function $\mu_N \colon \qdiff{E'}{E}{q} \to \MC_M$ defined as in Lemma~\ref{lem_QM} gives
  \begin{equation*}
    \mu_N(\langle 001 \rangle) = \mu_N(\langle 010 \rangle) = \mu_N(\langle 101 \rangle) = \mu_N(\langle 110 \rangle)
                             = \left\{\begin{amatrix} 100\\ 011 \end{amatrix}\right\} \in \MC_M.
  \end{equation*}
  Furthermore, $\mu_N$ satisfies \ref{QM} because
  \begin{align*}
    \langle 001 \rangle + \begin{amatrix} 100\\ 011 \end{amatrix} = \langle 010 \rangle + \begin{amatrix} 100\\ 011 \end{amatrix}
                                                   = \langle 101 \rangle + \begin{amatrix} 100\\ 011 \end{amatrix}
                                                   = \langle 110 \rangle + \begin{amatrix} 100\\ 011 \end{amatrix}.
  \end{align*}
\end{eg}

\subsection{Construction of Extensions}

We introduce a class of functions each of which gives rise to a unique extension.

\begin{dfn}
  For any $q$-matroid $M$ on $E$, a function $\mu \colon \qdiff{E'}{E}{q} \to \MC_M$ is called a \emph{modular cut selector} of $M$
  if $\mu$ satisfies \ref{QM}.
  We denote by $\MCS_M$ the set of all modular cut selectors of $M$.
\end{dfn}

The following theorem is one of the main results of this paper, which is a $q$-analogue of \cite[Theorem 7.2.3]{Oxley}.

\begin{thm}\label{thm_modular_cut_extension}
  Let $M=(E, r_M)$ be a $q$-matroid on $E$, $\mu$ be a modular cut selector of $M$, and $E' \geq E$ satisfy $\dim E' = \dim E + 1$.
  We define $r_{\mu} \colon \mathcal{L}(E') \to \mathbb{Z}_{\geq 0}$ as, for all $X \in \mathcal{L}(E)$ and $e \in \qdiff{E'}{E}{q}$,
  \begin{equation*}
    r_{\mu}(X) = r_M(X) \quad \text{and} \quad
    r_{\mu}(X + e) = r_M(X) + \delta_{X, e},
  \end{equation*}
  where $\delta_{X, e} = 0$ if $\cl_M(X) \in \mu(e)$; $\delta_{X, e} = 1$ otherwise.
  Then, $N_{\mu} = (E', r_{\mu})$ is a unique one-dimensional extension of $M$ such that for all $e \in \qdiff{E'}{E}{q}$,
  \begin{equation*}
    \{F \in \mathcal{F}_M \mid F + e \in \mathcal{F}_M \textrm{ and } r_{\mu}(F + e) = r_{\mu}(F)\} = \mu(e).
  \end{equation*}
\end{thm}

The proof of this theorem employs some lemmas.
We make a few observations on $r_{\mu}$ before going on to prove them.

\begin{rem}
  Although `$r_N$' in \cite[Theorem~7.2.3]{Oxley} is clearly well-defined. We need to check the well-definedness of $r_{\mu}$ here.
  The domain of the function $r_{\mu}$ above is indeed $\mathcal{L}(E')$
  since for all $Z \in \mathcal{L}(E')$, there exist $X\in \mathcal{L}(E)$ and $e\in [E']_q$ such that $Z = X + e$.
  Furthermore, if there also exist $Y \in \mathcal{L}(E)$ and $e'\in [E']_q$ such that $Z = Y + e'$,
  we obtain $X = Z \cap E = Y$. Hence, $\cl_M(X) + e' = \cl_M(X) + e$,
  which implies that the value $r_{\mu}(Z)$ is uniquely determined by \ref{QM}.
\end{rem}

\begin{lem}\label{modR1R2}
  $r_{\mu}$ in \textnormal{Theorem~\ref{thm_modular_cut_extension}} satisfies \ref{q-mat_boundedness} and \ref{q-mat_monotonicity}.
\end{lem}

\begin{proof}
  Consider first \ref{q-mat_boundedness}.
  Let $X \in \mathcal{L}(E)$ and $e \in \qdiff{E'}{E}{q}$.
  It is clear that $0 \leq r_{\mu}(X) = r_M(X) \leq \dim X$ and
  \begin{align*}
    0 \leq r_{\mu}(X + e) \leq r_M(X) + 1 \leq \dim{X}+1 = \dim(X + e).
  \end{align*}
  Next, we confirm \ref{q-mat_monotonicity}. Suppose $X \leq Y \in \mathcal{L}(E')$.
  \ref{q-mat_monotonicity} is clear when $Y \in \mathcal{L}(E)$, and so we may assume that $Y \notin \mathcal{L}(E)$.
  By the definition of $r_{\mu}$, $r_{\mu}(Y)\geq r_M(Y \cap E)$ because $\dim Y = \dim (Y\cap E) + 1$. If $X \in \mathcal{L}(E)$, then we have
  \begin{equation*}
    r_{\mu}(Y) - r_{\mu}(X) \geq r_M(Y \cap E) - r_M(X) \geq 0.
  \end{equation*}
   If $X \notin \mathcal{L}(E)$, there exists $e \in \qdiff{E'}{E}{q}$ such that $(X \cap E) + e = X$ and $(Y \cap E) + e = Y$.
   Noting that $X \cap E \leq Y \cap E$ and $\cl_M(X \cap E) \leq \cl_M(Y \cap E)$, we have
  \begin{align*}
    r_{\mu}(Y) - r_{\mu}(X) &= r_M(Y \cap E) - r_M(X \cap E) + \delta_{Y \cap E, e} - \delta_{X \cap E, e}\\
                    &= r_M(\cl_M(Y \cap E)) - r_M(\cl_M(X \cap E)) + \delta_{Y \cap E, e} - \delta_{X \cap E, e}.
  \end{align*}
  We assume for contradiction that $r_{\mu}(Y)-r_{\mu}(X) < 0$.
  Then $r_M(\cl_M(X \cap E)) = r_M(\cl_M(Y \cap E))$ and $ (\delta_{X \cap E, e}, \delta_{Y \cap E, e}) = (1, 0)$.
  We have $\cl_M(X \cap E) \notin \mu(e)$ and $\cl_M(Y \cap E) \in \mu(e)$, which contradict the definition of $\cl_M$ because $\cl_M(X \cap E) = \cl_M(Y \cap E)$.
  Therefore we obtain $r_{\mu}(Y) \geq r_{\mu}(X)$.
\end{proof}

\begin{lem}\label{modR3}
  $r_{\mu}$ in \textnormal{Theorem \ref{thm_modular_cut_extension}} satisfies \ref{q-mat_submodularity}.
\end{lem}

\begin{proof}
  Let $X, Y \in \mathcal{L}(E)$ and $e, e' \in \qdiff{E'}{E}{q}$.
  We only need to check that the inequality in \ref{q-mat_submodularity} holds for all pairs of subspaces having one of the forms $(X + e, Y)$, $(X + e, Y + e)$, and $(X + e, Y + e')$.

  In the cases $(X + e, Y)$ and $(X + e, Y + e)$, the property \ref{q-mat_submodularity}
  is similarly proved by considering a $q$-analogue of the proof in \cite[Theorem 7.2.3]{Oxley}, where we replace `$r_N$', `$\mathcal{M}$', `$\delta_X$', and `$\delta_Y$' in its proof with $r_\mu$, $\mu(e)$, $\delta_{X, e}$, and $\delta_{Y, e}$, respectively.
  \cite[Inequalities (7.2), (7.4), (7.5) and (7.6)]{Oxley} are obtained by \ref{q-mat_submodularity} for $r_M$ and both of the equations
  $(X + e) \cap Y = X \cap Y$ and $(X + e) \cap (Y + e) = X \cap Y + e$.

  We consider the case $(X + e, Y + e')$.
  It holds that
  \begin{equation*}
    \dim((X + e) \cap (Y + e')) \leq \dim(X \cap Y) + 1
  \end{equation*}
  by basic linear algebra.
  If $\dim((X + e) \cap (Y + e')) = \dim(X \cap Y) + 1$, there exists $e'' \in \qdiff{E'}{E}{q}$
  such that $(X + e) \cap (Y + e') = X \cap Y + e''$, and thus $X + e = X + e''$ and $ Y + e' = Y + e''$.
  So, we may assume $(X + e) \cap (Y + e') = X \cap Y$.
  This equation and \ref{q-mat_submodularity} for $r_M$ imply that
  \begin{align*}
          r_{\mu}(X + e) + r_{\mu}(Y + e') &=r_M(X) + \delta_{X, e} + r_M(Y) + \delta_{Y, e'} \\
                                   &\geq r_M(X + Y) + \delta_{X, e} + \delta_{Y, e'} + r_{\mu}((X + e) \cap (Y + e')).
  \end{align*}
  Moreover, there exists $\tilde{e} \in \qdiff{(X + Y + e + e') \cap E}{X+Y}{q}$ such that
  \begin{equation*}
      X + Y + e + \tilde{e} = X + Y + e' + \tilde{e} = X + Y + e + e',
  \end{equation*}
  since $\dim((X + Y + e + e') \cap E) = \dim(X + Y + e + e') - 1$.
  Thus, it suffices to show that
  \begin{equation*}
    r_M(X + Y) + \delta_{X, e} + \delta_{Y, e'} \geq r_{\mu}(X + Y + \tilde{e} + e).
  \end{equation*}
  If $\delta_{X, e} = \delta_{X, e'} = 0$, we have $\delta_{X + Y, e} = \delta_{X + Y, e'} = 0$ by \ref{M1}.
  Then, the property \ref{QM} implies that
  \begin{equation*}
    \cl_M(X + Y) + e = \cl_M(X + Y) + e' = \cl_M(X + Y) + e + e'.
  \end{equation*}
  So, we obtain
  \begin{equation*}
    r_{\mu}(X + Y) = r_{\mu}(X + Y + e) = r_{\mu}(X + Y + e + e').
  \end{equation*}
  If $\delta_{X, e} = \delta_{X, e'} = 1$, then $\delta_{X + Y + \tilde{e}, e} = 1$ by \ref{M1}. By Lemma~\ref{rXAB}\ref{AxA} and the definition of $r_{\mu}$,
  \begin{equation*}
    r_M(X + Y) + 2 \geq r_M(X + Y + \tilde{e}) + 1 \geq r_{\mu}(X + Y + \tilde{e} + e).
  \end{equation*}
  Without loss of generality, we may finally assume $\delta_{X, e} = 0$ and $\delta_{Y, e'} = 1$.
  Then, $\delta_{X + Y + \tilde{e}, e} = 0$ by \ref{M1}, and hence,
  \begin{equation*}
    r_M(X + Y) + 1 \geq r_M(X + Y + \tilde{e}) \geq r_{\mu}(X + Y + \tilde{e} + e).
  \end{equation*}
  Therefore, $r_{\mu}$ satisfies \ref{q-mat_submodularity} in all cases.
\end{proof}

\begin{proof}[Proof of Theorem~\ref{thm_modular_cut_extension}]
  $N_{\mu} = (E', r_{\mu})$ in the theorem is indeed a $q$-matroid by Lemmas~\ref{modR1R2} and \ref{modR3}, and its restriction to $E$ is equal to $M$.
  By Lemma~\ref{lem_QM}, we obtain the modular cut selector $\mu_{N_{\mu}}$ such that, for all $e \in \Delta_q(E', E)$,
  \begin{equation*}
    \mu_{N_{\mu}}(e) = \{F \in \mathcal{F}_M \mid F + e \in \mathcal{F}_{N_{\mu}} \textrm{ and } r_{\mu}(F + e) = r_{\mu}(F)\}.
  \end{equation*}
  Fix $e \in \qdiff{E'}{E}{q}$. By the definition of $r_{\mu}$,
  \begin{equation*}
    \mu(e) = \{F \in \mathcal{F}_M \mid \delta_{F, e} = 0\} = \{F \in \mathcal{F}_M \mid r_{\mu}(F + e) = r_{\mu}(F)\}.
  \end{equation*}
  Then, if a flat $F$ of $M$ satisfies $r_{\mu}(F + e) = r_{\mu}(F)$, we have $\cl_{N_{\mu}}(F) = F + e \in \mathcal{F}_{N_{\mu}}$ by Proposition~\ref{cor_rest}.
  This implies that
  \begin{equation*}
    \mu(e) = \{F \in \mathcal{F}_M \mid F + e \in \mathcal{F}_{N_{\mu}} \textnormal{ and } r_{\mu}(F + e) = r_{\mu}(F)\} = \mu_{N_{\mu}}(e).
  \end{equation*}

  If $N$ is a one-dimensional extension of $M$ and $\mu_N$ be a modular cut selector of $M$ defined in Lemma~\ref{lem_QM}, then $N = N_{\mu_N}$ by the definition of $r_{{\mu_N}}$.

  These two facts imply that the construction of modular cut selectors defined in Lemma~\ref{lem_QM} is the inverse of the construction of one-dimensional extensions defined in Theorem~\ref{thm_modular_cut_extension}.
  Hence, the uniqueness holds as required.
\end{proof}

\begin{cor}\label{cor_main}
  Let $M$ be a $q$-matroid.
  All the one-dimensional extensions of $M$ to $E'$ are obtained by applying the construction in Theorem~\textnormal{{\ref{thm_modular_cut_extension}}} to the collection of all modular cut selectors of $M$.
\end{cor}

\begin{eg}\label{trivial_extension}
  Let $M=(E_M, r_M)$ be a $q$-matroid and $N$, $N=(E_N, r_N)$ be its one-dimensional extension.
  We also let $\mu$ be the modular cut selector determined by $N$ and $M$.
  It holds that $r_N(E_N) = r_M(E_M + e)$ for all $e \in \Delta(E_N, E_M)$. 
  Theorem~\ref{thm_modular_cut_extension} indicate $r_N(E_N) = r_M(E_M) + 1$ if and only if $E_M \notin \mu(e)$ and this leads $\mu(e) = \emptyset$ by \ref{M1}.
  Then, we call $N$ the \emph{trivial extension} of $M$ and $\mu$ be the \emph{trivial modular cut selector} of $M$.
\end{eg}

  \section{Classifications of $q$-Matroids}\label{Results}
  It is challenging problem to classify all the $q$-matroids for given dimension $n$ due to the size of general linear group $\GL(n, q)$. 
The problem is solved only for $q$-matroids $M$ of arbitrary $n$, its rank $r(M) \leq 1$ or $r(M) \geq n-1$, 
where the parameter $q$ does not affect such classification \cite[Theorem 58]{DSOM}.
Concretely, the numbers of non-isomorphic $n$-dimensional $q$-matroids of rank $0$ and $1$ are 1 and $n$, respectively.

In this section, we propose an enumeration algorithm for $q$-matroids.
As a result we classify the $q$-matroids of dimension up to $5$ over $\mathbb{F}_2$ and up to  $4$ over $\mathbb{F}_3$.

\subsection{Canonical Representatives}\label{Canonical}
One of the most simple ways to classify $q$-matroids is to generate all the $q$-matroids 
and remove the isomorphic ones by pairwise isomorphism testing.
However, this approach is quite inefficient when the number of the generated $q$-matroids are large. 
We define a canonical presentation of $q$-matroid, which is the representative of each isomorphism class of $q$-matroids.
This is a $q$-analogue of the notion called \emph{canonical matroids} in \cite{MEIG}, and save the trouble of the pairwise isomorphism testing.

We employ $\mathbb{F}_q^n\coloneqq \prod_{i=1}^n{\mathbb{F}_q}$, 
the $n$-dimensional subspace over $\mathbb{F}_q$.
Fix a sufficiently large $u$, and the $i$-dimensional space $U^{(i)}:=\{\bm{v}=(v_1, \dots v_u) \in \mathbb{F}_q^u \mid v_j=0\; \textrm{for all}\; j > i\}$ for each integer $i=0,\cdots,u$.
Then, we obtain the chain $\{(0, \dots, 0)\}=U^{(0)} \leq U^{(1)} \leq \dots \leq U^{(u)} = \mathbb{F}_q^u$.

\begin{dfn}
  Let $X$ be a totally orderd set and its relation be $<$. 
  We also let $n$ be an integer.
  For distinct $\bm{x} = (x_1, \dots, x_{n}), \bm{y} = (y_1, \dots, y_{n}) \in X^l$, 
  we write $\bm{x} < \bm{y}$
  if $x_n < y_n$ or if $x_n = y_n$ and $(x_1, \dots, x_{n-1}) < (x_1, \dots, x_{n-1})$, 
  which we call the \emph{reverse lexicographic order}.
\end{dfn}

\begin{dfn}
  Let $A$ be a matrix on $\mathbb{F}_q$. 
  If the transpose of $A$ is the reduced row echelon form, 
  we say $A$ is of the \emph{reverse echelon form}.
  For a subspace $X$ of $E$, we call the generator matrix $B$ of $X$ the \emph{reverse canonical representation} of $X$ 
  if the matrix $B$ is of the reverse echelon form.
\end{dfn}

We fix a total order $\leq$ on 
$\mathbb{F}_q = \{\alpha_0 = 0 < \alpha_1 = 1 < \alpha_2 < \dots < \alpha_{q-1}\}$.
The order $<$ determin the reverse lexicographic order in the vectors of $\mathbb{F}_q^u$.
Hence, distinct $n$ times $m$ matrices over $\mathbb{F}_q$ also have the reverse lexicographic order $<$.
For all subspaces in $E$, we define the order $\prec$ 
by the reverse rexicographic order in the reverse canonical representation of them.

\begin{eg}
  We fix $u=4$ and consider $U^{(2)}, U^{(3)}$ over $\mathbb{F}_2$. 
  Let $P$ be the first $\gaussbinom{3}{2}{2}=7$ members of $\gaussbinom{U^{(4)}}{2}{2}$: 
  \begin{equation*}
    \begin{amatrix}
      1000\\
      0100
    \end{amatrix}
    \prec
    \begin{amatrix}
      1000\\
      0010
    \end{amatrix}
    \prec
    \begin{amatrix}
      0100\\
      0010
    \end{amatrix}
    \prec
    \begin{amatrix}
      1100\\
      0010
    \end{amatrix}
    \prec
    \begin{amatrix}
      0100\\
      1010
    \end{amatrix}
    \prec
    \begin{amatrix}
      1100\\
      1010
    \end{amatrix}
    \prec
    \begin{amatrix}
      1000\\
      0110
    \end{amatrix}.
  \end{equation*}
  Then, $P=\gaussbinom{U^{(2)}}{2}{2}$ as sets. 
  Let $P'$ be the $\gaussbinom{3}{2}{2}=7$ members 
  from $(\gaussbinom{3}{3}{2}+1)$th 
  to $(\gaussbinom{3}{3}{2}+\gaussbinom{3}{2}{2})$th 
  of $\gaussbinom{U^{(4)}}{3}{2}$:
    \begin{equation*}
    \begin{amatrix}
      1000\\
      0100\\
      0001
    \end{amatrix}
    \prec
    \begin{amatrix}
      1000\\
      0010\\
      0001
    \end{amatrix}
    \prec
    \begin{amatrix}
      0100\\
      0010\\
      0001
    \end{amatrix}
    \prec
    \begin{amatrix}
      1100\\
      0010\\
      0001
    \end{amatrix}
    \prec
    \begin{amatrix}
      0100\\
      1010\\
      0001
    \end{amatrix}
    \prec
    \begin{amatrix}
      1100\\
      1010\\
      0001
    \end{amatrix}
    \prec
    \begin{amatrix}
      1000\\
      0110\\
      0001
    \end{amatrix}.
  \end{equation*}
  Then, $P'$ consists of the first seven $3$-dimensional subspaces of $E'$ not contained in $E$. 
  The mapping $P' \to P$ which removes the the last row $(0001)$ preserves the canonicality and the order.
\end{eg}

\begin{dfn}
  Let $M$ be a $q$-matroid of rank $k$ on the $n$-dimensional space $U^{(n)}$, and set $K=\gaussbinom{n}{k}{q}$.
  We define the \emph{encoding} $\enc(M) \coloneqq (c_1, \dots, c_K) \in  \{0, 1\}^K$ of $M$ as follows:
  for all $i=1, \dots, K$, 
  if the $i$-th element of the $k$-dimensional subspaces of $E$ is a basis of $M$, 
  then $c_i = 1$, otherwise $c_i = 0$.
\end{dfn}

Note that 
the cryptomorphisms of $q$-matroids \cite{CNQC} guarantee that 
a collection of bases uniquely determines a $q$-matroid. 
In other words, each encoding uniquely determines a $q$-matroid.

\begin{proof}
  It is clear from the definition of bases and the restriction.
\end{proof}

\begin{dfn}
  Let $M$ be a $q$-matroid on $U^{(n)}$. 
  If $\enc(M)$ is the smallest encodding among those of $q$-matroids isomorphic to $M$ in the lexicographic order, where $0 < 1$. 
  We say $M$ is the \textit{canonical representative} for the isomorphism class of $M$, or simply, $M$ is \emph{canonical} if there is no risk of confusion.
\end{dfn}

Using the following lemma, we show that the canonical representative of a $q$-matroid can be obtained from that of its one-dimensional extension. 
\begin{lem}\label{lem_base}
  Let $M$ be a $q$-matroid and $N$ be its one-dimensional extension. 
  If $N$ is not the trivial extension of $M$, the bases $\mathcal{B}_M$ of $M$ satisfies
  $$
  \mathcal{B}_{M} = \{B \in \mathcal{B}_N\mid B \leq E_M\}, 
  $$
  othwerwise,
  $$
  \mathcal{B}_{M} = \{B \cap E_M\mid B \in \mathcal{B}_N \}.
  $$
\end{lem}
\begin{proof}
  It is clear from the definition of bases and the restriction.
\end{proof}

\begin{thm}
  Let $M$ be a $q$-matroid of rank $k$ on the ground space $U^{(n)}$. 
  Let $N$ be its one-dimensional extension to $U^{(n+1)}$. 
  When $N$ is the canonical, the followings hold:
  \begin{enumerate}[label=\textnormal{(\arabic*)}]
    \item If $N$ is not the trivial extension of $M$, 
    then $M$ is encoded as the first $\gaussbinom{n}{k}{q}$ symbols of $\enc(N)$. 
    Furthermore, $M$ is also canonical.\label{thm_canonical_1}
    \item If $N$ is the trivial extension of $M$, then $M$ is encoded as the $\gaussbinom{n}{k}{q}$ symbols 
    from $(\gaussbinom{n}{k+1}{q}+1)$-th to $(\gaussbinom{n}{k+1}{q}+\gaussbinom{n}{k}{q})$-th of $\enc(N)$. 
    Furthermore, $M$ is also canonical.\label{thm_canonical_2}
  \end{enumerate}
\end{thm}
\begin{proof}
  We shall show (\ref{thm_canonical_1}). 
  By Lemma \ref{lem_base} and the definition of $\prec$, the bases of $M$
  are exactly the bases of $N$ and encoded as the first $\gaussbinom{n}{k}{q}$ symbols of $\enc(N)$.
  If $M$ is not the canonical, 
  there exists a $q$-matroid $M'$ isomorphic to $M$ whose encoding is smaller than that of $M$. 
  Further, there exists an extension $N'$ of $M'$ which is isomorphic to $N$.
  Hence, $\enc(N')$ is smaller than $\enc(N)$, 
  which contradicts the assumption that $N$ is the canonical.

  The statement (\ref{thm_canonical_2}) can be shown in the same way.
  We consider the $k$-dimensional subspaces in the ground space of $N$ 
  from $(\gaussbinom{n}{k+1}{q}+1)$-th 
  to 
  $(\gaussbinom{n}{k+1}{q}+\gaussbinom{n}{k}{q})$-th.
  We consider the reverse canonical representaion of these subspaces.
  When we delete each last row of these matrices,
  we have all the reverse canonical representations of $(k-1)$-dimensional subspaces of $E_M$ preserving the order $\prec$.
  By Lemma \ref{lem_base} and the definition of $\prec$, the bases of $M$
  are exactly those of $N$ and encoded as the $\gaussbinom{n}{k}{q}$ symbols 
  from $(\gaussbinom{n}{k+1}{q}+1)$-th to $(\gaussbinom{n}{k+1}{q}+\gaussbinom{n}{k}{q})$-th of $\enc(N)$.
  By the same argument as above, $M$ is also canonical.
\end{proof}
\begin{lem}
  Let $M$ be a $q$-matroid. If $M$ is canonical,
  then the trivial extension of $M$ is also canonical.
\end{lem}

\begin{proof}
  Let $M$ be the $q$-matroid $M$ of rank $k$ on $U^{(n)}$. 
  We also let $N$ be the trivial extension of $M$.
  The first $\gaussbinom{n}{k+1}{q}$ symbols of $\enc(N)$ are all $0$,
  and the following $\gaussbinom{n}{k}{q}$ symbols are the same as the encoding of $M$.
  We assume $N$ is not canonical, there exists a $q$-matroid $N'$ isomorphic to $N$ whose encoding is smaller than that of $N$.
  But then, the restriction $M'$ of $N'$ to $E$ is isomorphic to $M$ and its encoding is smaller than that of $M$, 
  which contradicts the assumption that $M$ is canonical.
\end{proof}

\subsection{Algorithm}
Let $M$ be a $q$-matroid.
For any modular cut selector $\mu$ of $M$, we say the $q$-matroid $N_{\mu}$ in Theorem~\ref{thm_modular_cut_extension} is \textit{induced} by $\mu$ from $M$.
We enumerate all the one-dimensional extensions to $E'$ from modular cut selectors.
For the purpose, we first construct all the modular cut selectors of $M$.
To find all the members of $\MC_M$, it suffices to consider the minimal members of each $\mathcal{M} \in \MC_M$ because we obviously have
\begin{equation*}
  \mathcal{M} = \{F \in \mathcal{F}_M \mid \text{there exists } F' \in m(\mathcal{M}) \textrm{ such that } F' \leq F \},
\end{equation*}
where
\begin{equation*}
  m(\mathcal{M})\coloneqq\{F \in \mathcal{M} \mid \text{for all } F' \in \mathcal{F}_M, \textrm{ if } F \leq F' \textrm{, then } F=F'\}.
\end{equation*}
We notice that $m(\mathcal{M})$ forms an \emph{anti-chain}, \textit{i.e.},
any two different elements in $m(\mathcal{M})$ is not in inclusion relation.
Conversely, given an anti-chain $\mathcal{A}$ whose members are flats of $M$, the subcollection of flats
\begin{equation*}
  \{F \in \mathcal{F}_M \mid \text{there exists } F' \in \mathcal{A} \textrm{ such that } F' \leq F \}
\end{equation*}
satisfies \ref{M1}.
Consequently, we can find $\mathcal{M}$ as a subset obtained from such anti-chains and satisfying \ref{M2}.

\begin{breakablealgorithm}
  \caption{The modular cuts of a $q$-matroid}
    \begin{algorithmic}
      \INPUT $q$-matroid $M$.
      \OUTPUT $\MC_M$.
      \State {Set $\mathcal{D}$ as an empty set.}
      \For {Anti-chain $\mathcal{A}$ in flats $\mathcal{F}_M$}
        \State {$\mathcal{M} \leftarrow \{F \in \mathcal{F}_M : \textrm{there exists $F' \in \mathcal{A}$ such that $F' \leq F$}\}$}
        \If {$\mathcal{M}$ satisfies \ref{M2}}
          \State {Add $\mathcal{M}$ to $\MC_M$.}
        \EndIf
      \EndFor
      \State {\Return $\MC_M$}.
    \end{algorithmic}
\end{breakablealgorithm}

To construct a modular cut selector $\mu$, 
for each $e \in \qdiff{E'}{E}{q}$, 
we select a member of $\MC_M$ as $\mu(e)$ so as not to violate \ref{QM}. 

Let $F$ be a flat of $M$ and $\qdiff{E'}{E}{q}= \{e_0 \prec e_1 \prec \cdots \prec e_{end}\}$.
We call $\mu$ \emph{admissible} in $\mu(e_0), \cdots, \mu(e_{i-1})$ if the next statement holds:
for all integers $0 \leq j < i$ and $F \in \mathcal{M}$, $e_j \leq F + e_i$ if and only if $ F \in \mu(e_j)$.
By the definition, 
a function $\mu$ is a modular cut selector of $M$ if and only if 
$\mu$ is admissible in $\mu(e_0), \cdots, \mu(e_{i-1})$ for each $i$.

Hence, we construct a function $\mu \colon \qdiff{E'}{E}{q} \to \MC_M$ by setting the $i$-th value $\mu(e_i)$
as a member $\mathcal{M}$ of $\MC_M$ in order.
Then, assigning the last value $\mu(e_{end})$ means $\mu$ satisfies \ref{QM}.

We summarize the discussion above and introduce the way of modular cuts enumeration in Algorithm~\ref{MODALG}.

\begin{breakablealgorithm}
    \caption{The Modular cut selectors of a $q$-matroid}\label{MODALG}
    \begin{algorithmic}
      \INPUT $\MC_M$.
      \OUTPUT $\MCS_M$.
      \Function{Modular-Cut-Selector}{$\MCS_M, \mathcal{M}, i$}
      \For {$\mathcal{M}$ in $\MC_M$}
          \If {$\mathcal{M}$ is admissible in $\mu(e_0), \dots, \mu(e_{i-1})$}
            \State $\mu(e_i) \gets \mathcal{M}$
            \If {$e_i = e_{end}$}
              \State Add $\mu$ to $\MCS_M$
            \Else
              \State $\MCS_M \gets \MCS_M \cup \Call{Modular-Cut-Selector}{\MCS_M, \mu, i+1}$
            \EndIf
          \EndIf
        \EndFor
        \State \Return $\MCS_M$
      \EndFunction
      \State Initialize $\MCS_M$ as an empty set.
      \State Define $\mu$ as a function on $\qdiff{E'}{E}{q}$, where $\mu(e) = -1$ for $e \in \qdiff{E'}{E}{q}$.
      \State \Comment{-1 means $\mu(e)$ has not been assigned.}
      \State \Return \Call{Modular-Cuts}{$\MCS_M, \mu, 0$}
    \end{algorithmic}
\end{breakablealgorithm}

By Example~\ref{trivial_extension}, 
the rank of a one-dimensional extension is uniquely determined by its modular cuts.
This facts enable us to classify $q$-matroids not only by dimension but also by rank.

Now, we obtain the classification method as in Algorithm~\ref{CLAALG}.
Let $\qMat(n, k)$ be a set of the canonical representatives of $q$-matroids of rank $k$ on $U^{(n)}$.

\begin{breakablealgorithm}
    \caption{Enumeration of non-isomorphic $q$-matroid}\label{CLAALG}
    \begin{algorithmic}
      \INPUT $\qMat(n-1, k)$ and $\qMat(n-1, k-1)$.
      \OUTPUT $\qMat(n, k)$.
      \For {$M$ in $\qMat(n-1, k)$}
        \For {non-trivial modular cut selector $\mu$ of $M$}
          \State {$N \leftarrow$ $q$-matroid induced by $\mu$.}
          \If {$N$ is the canonical}
            \State {{Add $N$ to $\qMat(n, k)$.}}
          \EndIf
        \EndFor
      \EndFor
      \For {$M$ in $\qMat(n-1, k-1)$}
        \State {Add trivial extension of $M$ to $\qMat(n, k)$.}
      \EndFor
      \State {\Return $\qMat(n, k)$.}
    \end{algorithmic}
\end{breakablealgorithm}

\subsection{Results}
We note that Proposition~\ref{iso_dual} guarantee that, in order to classify $q$-matroids,
it suffices to consider non-isomorphic $q$-matroids of dimension $n$ and rank $k \leq \lfloor \frac{n+1}{2} \rfloor$.

Using open source mathematics software system ``SageMath'' on processor ``Intel Core i9-13980HX'' (2.20 GHz),
we implemented our algorithms and classified $q$-matroids of dimension $4, 5$ over $\mathbb{F}_2$,
and those of dimension $4$ over $\mathbb{F}_3$. 
We summarize the results as follows.
\begin{thm}
  Over $\mathbb{F}_2$,
  there exist $10$ non-isomorphic $q$-matroids of dimension $4$ and
  $94$ non-isomorphic $q$-matroids of dimension $5$ for rank $2$.

  Over $\mathbb{F}_3$, there exist at most $31$ non-isomorphic $q$-matroids of dimension $4$ and rank $2$.
\end{thm}

\begin{table}[H]
  \centering
    \caption{The number of $q$-matroids on $\mathbb{F}_2$} \label{tbl:q-mat_F2}
    \begin{tabular}{r|cccccc}
      $k \backslash n$ & $1$ & $2$ & $3$ & $4$  & 5               & $\cdots$\\ \hline
      $0$ &   $1(\cite{DSOM})$   &  1(\cite{DSOM})  &  1(\cite{DSOM})  &  1(\cite{DSOM})   & 1(\cite{DSOM})               & -\\
      $1$ &   1(\cite{DSOM})   &  2(\cite{DSOM})  &  3(\cite{DSOM})  &  4(\cite{DSOM})   & 5(\cite{DSOM})       & -\\
      $2$ &       &  1(\cite{DSOM})  &  3(\cite{DSOM})  &  $\bf{10}$ & $\bf{94}$ & -\\
      $3$ &       &     &  1(\cite{DSOM})  &  4(\cite{DSOM})   & $\bf{94}$ & -\\
      $4$ &       &     &     &  1(\cite{DSOM})   & 5(\cite{DSOM})       & -\\
      $5$ &       &     &     &      & 1(\cite{DSOM})               & -\\
    \end{tabular}
\end{table}
\begin{table}[H]
  \centering
    \caption{The number of $q$-matroids on $\mathbb{F}_3$} \label{tbl:q-mat_F3}
    \begin{tabular}{r|ccccc}
      $k \backslash n$ & $1$ & $2$ & $3$ & $4$             &  $\cdots$\\ \hline
      $0$ &   1(\cite{DSOM})   &  1(\cite{DSOM})  &  1(\cite{DSOM})  &  1(\cite{DSOM})              &    -\\
      $1$ &   1(\cite{DSOM})   &  2(\cite{DSOM})  &  3(\cite{DSOM})  &  4(\cite{DSOM})              & - \\
      $2$ &       &  1(\cite{DSOM})  &  3(\cite{DSOM})  &  $\bf{31}$ &  -\\
      $3$ &       &     &  1(\cite{DSOM})  &  4(\cite{DSOM})              &  -\\
      $4$ &       &     &     &  1(\cite{DSOM})              &  -\\
      $5$ &       &     &     &                 &  -\\
    \end{tabular}
\end{table} 
  \section{Application}\label{q-FanoPlane}
  
\begin{dfn}
  Let $\mathcal{H}$ be a collection of the $k$-subspaces of an $n$-dimensional vector space $E$ over $\mathbb{F}_q$. 
  If each $t$-subspace of $E$ is contained in exactly one member of $\mathcal{H}$, 
  then $(E, \mathcal{H})$ is called a $q$-\emph{Steiner system} with parameters $S_q(t, k, n)$. 
  Further, We call $\mathcal{H}$ the \emph{blocks} of $(E, \mathcal{H})$.
\end{dfn}
One of the most famous open problems in $q$-analogue of combinatorics is the existence problem of $S_q(2, 3, 7)$ called the $q$-\emph{Fano plane}.

Residual $q$-Fano planes are the substructures of putative $q$-Fano plane $\mathcal{S}=(E, \mathcal{H})$ and there are two different definitions. 
One can be regarded as the blocks obtained by deleting last cordinate from the blocks of $q$-Fano plane defined in \cite{RFPRS}.
The explicit construction of such blocks is shown in the paper. 
However, it has not been revealed whether these blocks could be extended to $q$-Fano plane. 
The other can be regarded as the blocks of $q$-Fano plane contained in a fixed hyperplane \cite{DRSD}.

We use latter definition because related to the restriction of $q$-matroids due to the following fact.
\begin{thm}[\cite{CNMD}]\label{q-PMD}
  Let $\mathcal{S}=(E, \mathcal{H})$ be a $q$-Steiner system with parameters $S_q(t, k, n)$.
  We define $\mathcal{F}_{\mathcal{S}}$ as
  \begin{align*}
    \mathcal{F}_{\mathcal{S}} := \left\{\bigcap \mathcal{A} \mid \mathcal{A} \subseteq \mathcal{H}\right\},
  \end{align*}where the intersection of the empty set is $E$.
  Then $\mathcal{F}_{\mathcal{S}}$ is a flat of a $n$-dimensional $q$-matroid of rank $t+1$.
  Furthermore, for each integer $0 \leq i < t$, the collection of rank $i$-flats is the $i$-dimensional subspaces of $E$, 
  the collcetion of rank $t$-flats is precisely $\mathcal{H}$, 
  and the unique rank $(t+1)$-flat is $E$.
\end{thm}

We assume $\mathcal{S}$ is a putative $q$-Fano plane on a $7$-dimensional vector space $E$ over $\mathbb{F}_2$ 
and $M_{\mathcal{S}}$ is a $q$-matroid on $E$ in Theorem~\ref{q-PMD}.
If we fix a hyperplane $E_6$ of $E$, 
the 3-dimensional rank 2-flat of $M_{\mathcal{S}}|E_6$ is a residual $q$-Fano plane. 

Let $E_5$ be a $5$-dimensional subspace of $E$. 
In \cite{INSD}, it is demonstrated that
each block of $\mathcal{S}$ intersects in $256$, $120$, $5$ times with $E_5$
in $1$, $2$, $3$-dimensional subspaces respectively.
In addition, these $2$-dimensional subspaces are different because, 
if not, there are two blocks intersecting in a $2$-dimensional subspace, 
which contradicts the definition of $q$-Steiner system.
By Theorem~\ref{cor_rest} and definition of restriction, 
these $120+5$ intersection spaces are the rank $2$ flats of $M_{\mathcal{S}}|E_5$, and $M_{\mathcal{S}}|E_5$ is of rank $3$.

As a result of classification of $q$-matroids over $\mathbb{F}_2$, 
such 5-dimensional $q$-matroids are limited to following $10$ up to isomorphism.
We list the 3-dimensional rank 2 flats of them (their 2-dimensional rank 2 flats are the all 2-subspaces not contained in any 3-dimensional rank 2 spaces), 
the number of their intersection points, 
and the automorphism over $\GL(5, 2)$ of them.
\begin{equation*}
\end{equation*}
\begin{enumerate}
  \item
      \begin{align*}
        \left\{
        \begin{amatrix}10000\\01000\\00001\end{amatrix},
        \begin{amatrix}00100\\00010\\00001\end{amatrix},
        \begin{amatrix}10010\\01100\\00001\end{amatrix},
        \begin{amatrix}10110\\01010\\00001\end{amatrix},
        \begin{amatrix}10100\\01110\\00001\end{amatrix}
        \right\}
      \end{align*}
    \begin{itemize}
      \item Automorphism order: $5760$
      \item Intersection points: 1
    \end{itemize}

  \item
      \begin{align*}
        \left\{
        \begin{amatrix}10100\\01000\\00010\end{amatrix},
        \begin{amatrix}10000\\01000\\00001\end{amatrix},
        \begin{amatrix}00100\\00010\\00001\end{amatrix},
        \begin{amatrix}10001\\01010\\00100\end{amatrix},
        \begin{amatrix}10000\\01011\\00110\end{amatrix}
        \right\}
      \end{align*}
    \begin{itemize}
      \item Automorphism order: $5$
      \item Intersection points: 10
    \end{itemize}

  \item
      \begin{align*}
        \left\{
        \begin{amatrix}10100\\01000\\00010\end{amatrix},
        \begin{amatrix}10000\\01000\\00001\end{amatrix},
        \begin{amatrix}00100\\00010\\00001\end{amatrix},
        \begin{amatrix}10001\\01010\\00100\end{amatrix},
        \begin{amatrix}10010\\01011\\00110\end{amatrix}
        \right\}
      \end{align*}
    \begin{itemize}
      \item Automorphism order: $8$
      \item Intersection points: 10
    \end{itemize}

  \item
    \begin{align*}
      \left\{
      \begin{amatrix}10100\\01000\\00010\end{amatrix},
      \begin{amatrix}10000\\01000\\00001\end{amatrix},
      \begin{amatrix}00100\\00010\\00001\end{amatrix},
      \begin{amatrix}10010\\01011\\00100\end{amatrix},
      \begin{amatrix}10000\\01011\\00111\end{amatrix}
      \right\}
    \end{align*}
    \begin{itemize}
      \item Automorphism order: $120$
      \item Intersection points: 10
    \end{itemize}

  \item
    \begin{align*}
      \left\{
      \begin{amatrix}10100\\01000\\00010\end{amatrix},
      \begin{amatrix}10000\\01000\\00001\end{amatrix},
      \begin{amatrix}00100\\00010\\00001\end{amatrix},
      \begin{amatrix}10001\\01010\\00100\end{amatrix},
      \begin{amatrix}10010\\01001\\00110\end{amatrix}
      \right\}
    \end{align*}
    \begin{itemize}
      \item Automorphism order: $3$
      \item Intersection points: 10
    \end{itemize}

  \item
    \begin{align*}
      \left\{
      \begin{amatrix}10100\\01000\\00010\end{amatrix},
      \begin{amatrix}10000\\01000\\00001\end{amatrix},
      \begin{amatrix}00100\\00010\\00001\end{amatrix},
      \begin{amatrix}10001\\01010\\00100\end{amatrix},
      \begin{amatrix}10011\\01011\\00111\end{amatrix}
      \right\}
    \end{align*}
    \begin{itemize}
      \item Automorphism order: $12$
      \item Intersection points: 10
    \end{itemize}

  \item
    \begin{align*}
      \left\{
      \begin{amatrix}10100\\01000\\00010\end{amatrix},
      \begin{amatrix}10000\\01000\\00001\end{amatrix},
      \begin{amatrix}00100\\00010\\00001\end{amatrix},
      \begin{amatrix}10010\\01100\\00001\end{amatrix},
      \begin{amatrix}10011\\01010\\00100\end{amatrix}
      \right\}
    \end{align*}
    \begin{itemize}
      \item Automorphism order: $2$
      \item Intersection points: 8
    \end{itemize}

  \item
    \begin{align*}
      \left\{
      \begin{amatrix}10100\\01000\\00010\end{amatrix},
      \begin{amatrix}10000\\01000\\00001\end{amatrix},
      \begin{amatrix}00100\\00010\\00001\end{amatrix},
      \begin{amatrix}10010\\01100\\00001\end{amatrix},
      \begin{amatrix}10101\\01100\\00010\end{amatrix}
      \right\}
    \end{align*}
    \begin{itemize}
      \item Automorphism order: $8$
      \item Intersection points: 6
    \end{itemize}

  \item
    \begin{align*}
      \left\{
      \begin{amatrix}10100\\01000\\00010\end{amatrix},
      \begin{amatrix}10000\\01000\\00001\end{amatrix},
      \begin{amatrix}00100\\00010\\00001\end{amatrix},
      \begin{amatrix}10010\\01100\\00001\end{amatrix},
      \begin{amatrix}10001\\01010\\00100\end{amatrix}
      \right\}
    \end{align*}
    \begin{itemize}
      \item Automorphism order: $2$
      \item Intersection points: 8
    \end{itemize}

  \item
    \begin{align*}
      \left\{
      \begin{amatrix}10100\\01000\\00010\end{amatrix},
      \begin{amatrix}10000\\01000\\00001\end{amatrix},
      \begin{amatrix}00100\\00010\\00001\end{amatrix},
      \begin{amatrix}10010\\01100\\00001\end{amatrix},
      \begin{amatrix}10100\\01110\\00001\end{amatrix}
      \right\}
    \end{align*}
    \begin{itemize}
      \item Automorphism order: $48$
      \item Intersection points: 5
    \end{itemize}
\end{enumerate}

\begin{thm}
  Assume that $\mathcal{S}$ is a putative $q$-Fano plane $S_2(2, 3, 7)$.
  Then the $q$-matroid $M_{\mathcal{S}}$ in Theorem\ref{q-PMD} is an extension of a $q$-matroid
  which is isomorphic to at least one of $(1) \dots (10)$, 
  that is, 
  none of the extensions of the other $84$ representatives of rank $3$ and dimension $5$ are isomorphic to $M_{\mathcal{S}}$.
\end{thm} 
  \section*{Acknowledgements}
   This work was supported by the Japan Society for the Promotion of Science (JSPS) KAKENHI Grants JP21J22027, JP25K17298 and the Japan Science and Technology Agency (JST) BOOST Grants JPMJBS2408.

\end{document}